\documentclass[12pt,a4paper]{article}  
\usepackage[margin=1in]{geometry}
\usepackage{titlesec}
\titleformat{\section}{\normalsize\bfseries}{\thesection}{1em}{}
\titleformat{\subsection}{\normalsize\bfseries}{\thesubsection}{1em}{}
\titleformat{\subsubsection}{\small\bfseries}{\thesubsubsection}{1em}{}
\usepackage{amsmath, amssymb, amsfonts, mathrsfs}  
\usepackage{amsthm}       
\usepackage{mathtools}     
\usepackage{thmtools}     

\usepackage{graphicx}     
\usepackage{float} 
\usepackage{tikz}        
\usepackage{caption}      
\usepackage{subcaption}    
\usepackage{epstopdf}      

\usepackage[numbers, sort&compress]{natbib}      
\usepackage{xcolor}
\definecolor{mycyanblue}{HTML}{0080AC}
\usepackage{hyperref}
\hypersetup{
    colorlinks=true,
    allcolors=mycyanblue
}

\usepackage{booktabs}      
\usepackage{array}         
\usepackage{multirow}      
\usepackage{enumitem}     

\theoremstyle{plain}       
\newtheorem{theorem}{Theorem}[section]   
\newtheorem{lemma}[theorem]{Lemma}
\newtheorem{corollary}[theorem]{Corollary}
\newtheorem{proposition}[theorem]{Proposition}

\theoremstyle{definition}

\newtheorem{remark}[theorem]{Remark}

\title{Main-Factor Allocation and Coronal Realizability for Generalized Cospectral Mates of Trees}
\author{Chaochao Zhu$^{1}$ \\
    $^{1}$College of Finance and Mathematics, West Anhui University, \\ Lu'an 237012, China \\
    \texttt{zccjsbz@amss.ac.cn}} 
\date{}  

\begin{document}

\maketitle

\begin{abstract}
We study how irreducible factors associated with main eigenvalues constrain the connected components of generalized cospectral mates. Let $M_G(x)$ be the main polynomial of a graph $G$, and let $\kappa_{\mathrm m}(G)$ denote the sum of the multiplicities in $\phi_G(x)$ of the irreducible factors dividing $M_G(x)$. We prove that every graph $H$ with the same characteristic polynomial and main polynomial as $G$ satisfies \(c(H)\leq \kappa_{\mathrm m}(G).\) Consequently, if $T$ is a tree and $H$ is generalized cospectral with $T$, then \(\beta(H)=c(H)-1\leq \kappa_{\mathrm m}(T)-1.\) We develop the case $\kappa_{\mathrm m}(T)=2$ in detail. Any disconnected generalized cospectral mate is the union of a tree and a connected bipartite unicyclic graph, and the coronals of its two components are uniquely prescribed by the canonical decomposition of the coronal of $T$ with respect to the two irreducible main factors. This turns the existence of a disconnected mate into a component-realizability problem. Factor moments yield realizability and tree-forcing obstructions, while matching data provide complementary characteristic-polynomial information. In particular, the unique cycle of any disconnected mate has length at least $6$, and finitely many matching identities, together with component-coronal realizability, certify generalized cospectrality under an explicit degree bound. As an application, we show that the double star $D(2m,m+1)$ is DGS but not DS whenever $m\geq2$ and neither $m$ nor $2m+2$ is a perfect square. In particular, $D(8t+4,4t+3)$, $t\geq0$, gives an explicit infinite family of such graphs.
\end{abstract}


\section{Introduction}\label{sec:introduction}
Whether a graph is determined by its adjacency spectrum is a basic question in spectral graph theory. A graph is said to be \emph{determined by its spectrum} (DS) if every graph with the same adjacency spectrum is isomorphic to it; see \cite{vanDamHaemers2003,vanDamHaemers2009,BrouwerHaemers2012}. For trees, adjacency cospectrality is abundant: Schwenk proved that almost all trees have a nonisomorphic cospectral mate \cite{Schwenk1973}.

The \emph{generalized spectrum} of a graph consists of its adjacency spectrum together with that of its complement. A graph is \emph{determined by its generalized spectrum} (DGS) if every graph with the same generalized spectrum is isomorphic to it. Two graphs with the same generalized spectrum are called \emph{generalized cospectral}. For a graph $G$ of order $n$, let $A(G)$ denote its adjacency matrix, \(\phi_G(x):=\det(xI-A(G)),\) its characteristic polynomial, and let $\mathbf1$ denote the all-one
vector.

The rank-one perturbation by the all-one matrix underlying generalized cospectrality goes back at least to Johnson and Newman \cite{JohnsonNewman1980}; see also
\cite{vanDamHaemersKoolen2007}. Many DGS criteria are formulated in terms of the walk matrix
\(
W(G)=
[\,\mathbf1,A(G)\mathbf1,\ldots,A(G)^{n-1}\mathbf1\,]
\)
and arithmetic properties of rational orthogonal matrices \cite{WangXu2006,Wang2013,Wang2017,MaoLiuWang2015, LiuSiemonsWang2019,QiuEtAl2023,WangWangZhu2023}. For trees, Yang and Wang showed that irreducibility of the characteristic polynomial over $\mathbb Q$ removes the usual oddness and square-freeness assumptions from a walk-matrix criterion
\cite{YangWang2024}. Lin, Allem, Trevisan, Wang and Zhang later constructed an infinite family of double stars that are DGS but not DS; irreducibility of the nonzero characteristic factor is central to their argument \cite{LinEtAl2026}.

The approach here is different. An adjacency eigenvalue $\lambda$ of $G$ is called \emph{main} if its eigenspace is not orthogonal to $\mathbf1$. Main eigenvalues and their connections with walk matrices and graph structure have been studied extensively; see, for example, Hagos \cite{Hagos2002}. If $\lambda_1,\ldots,\lambda_s$ are the distinct main eigenvalues, define the main polynomial by
\[
M_G(x):=\prod_{i=1}^{s}(x-\lambda_i).\]
Main and related spectral polynomials have appeared in several contexts \cite{ScirihaFarrugia2011,CardosoGomesPinheiro2022}; see also the complementary plain polynomial in
\cite{WangWang2025}.

We use the irreducible factors of $M_G$ as component data. Write \(\phi_G(x)=\prod_{i=1}^{r}f_i(x)^{m_i},\) where the $f_i$ are distinct monic polynomials irreducible over $\mathbb Q$, and set
\[
\kappa_{\mathrm m}(G)
:=\sum_{f_i\mid M_G}m_i.
\]
Our first result is the component bound \(c(H)\leq \kappa_{\mathrm m}(G)\) for every graph $H$ having the same characteristic polynomial and the same main polynomial as $G$. The proof rests on a simple allocation principle. The Perron root of each connected component of $H$ is main in $H$; its minimal polynomial is therefore one of the
irreducible factors of the common main polynomial and must occur in the characteristic polynomial of that component. The multiplicities of the main factors thus form a finite supply that must be distributed among the components.

For generalized cospectral mates of a tree this has an immediate topological consequence. If $T$ is a tree and $H$ is generalized cospectral with $T$, then $H$ is bipartite and has the same numbers of vertices and edges as $T$. Hence, with \(\beta(H):=|E(H)|-|V(H)|+c(H),\) one has
\[
\beta(H)=c(H)-1
\leq \kappa_{\mathrm m}(T)-1.
\]
In particular, $\kappa_{\mathrm m}(T)=1$ forces every generalized cospectral mate of $T$ to be a tree. For a connected graph,
\[
\kappa_{\mathrm m}(G)=1
\quad\Longleftrightarrow\quad
M_G(x)\ \text{is irreducible over }\mathbb Q.
\]
Thus irreducibility of the main polynomial, rather than of the full characteristic polynomial, already gives a tree-forcing criterion.

The first nontrivial case is $\kappa_{\mathrm m}(T)=2$. We show that any disconnected generalized cospectral mate has the form \(H=R\cup U,\) where $R$ is a tree and $U$ is connected, bipartite, and unicyclic. Moreover, \(M_T(x)=f(x)g(x),\) where $f$ and $g$ are distinct monic polynomials irreducible over $\mathbb Q$, each occurring with exponent one in $\phi_T$, and, after interchanging $f$ and $g$ if necessary, \(M_R=f, M_U=g.\)

The graph coronal now becomes decisive. Recall that
\(
\Gamma_G(x)
=\mathbf1^{\mathsf T}(xI-A(G))^{-1}\mathbf1,
\)
introduced by McLeman and McNicholas \cite{McLemanMcNicholas2011}. Its poles are precisely the main eigenvalues, its reduced denominator is $M_G$, and its Laurent
expansion at infinity records the total walk counts. Since $M_T=fg$, coprimality of $f$ and $g$ yields a unique canonical decomposition \(\Gamma_T=\Gamma_f+\Gamma_g\)
into reduced proper fractions with denominators $f$ and $g$. If $H=R\cup U$ is a disconnected generalized cospectral mate, then necessarily \(\Gamma_R=\Gamma_f, \Gamma_U=\Gamma_g.\) Thus the two summands are not merely formal partial fractions: they must be coronals of actual connected components. The existence of a disconnected mate is thereby reduced to a component-realizability problem.

The Laurent coefficients of the factor coronals give immediate obstructions. If $\Gamma_h$ is realized by a graph $X$, then its Laurent coefficients are the total walk counts of $X$. They must therefore be nonnegative integers and satisfy the low-order identities forced by the order, size, and degree sequence of $X$. In the two-factor case one prospective component is a tree and the other is unicyclic, so their Euler defects impose different moment conditions. These conditions yield effective tree-forcing criteria; in particular, a single nonintegral factor moment may already exclude every disconnected generalized cospectral mate.

The coronal does not, however, determine the entire characteristic polynomial. Matching data supply the missing information. Suppose that an actual disconnected generalized cospectral mate is $H=R\cup U$, and let $C_{2\ell}$ be the unique cycle of $U$. The cycle contribution to the determinant gives an exact correction formula for the matching numbers of $T$ and $H$. In particular, $2\ell\geq6$ and
\[
\nu(T)-\nu(R)\geq \ell-1\geq2,
\]
where $\nu$ denotes the matching number. We also obtain corresponding nullity and maximum-matching restrictions.

The same matching relations work in the reverse direction. Suppose that the two canonical factor coronals are realized by a tree $R$ and a connected bipartite unicyclic graph $U$, and put $H=R\cup U$. If the first $s+1$ matching identities arising from the cycle correction hold and
\[
d:=\deg\frac{\phi_T}{M_T}
\leq 2s+1,
\]
then they force $\phi_T=\phi_H$.

The factor-coronal realization already gives equality of the coronals, and hence $T$ and $H$ are generalized cospectral. Thus the two parts of the method are complementary: factor coronals determine the main spectral information componentwise, while finitely many matching coefficients recover the remaining characteristic-polynomial data.

The resulting mechanism is
\[
\text{main-factor allocation}
\longrightarrow
\text{component-coronal realizability}
\longrightarrow
\text{finite matching certification}.
\]
The first step bounds and allocates the components; the second tests their realizability; the third supplies spectral information not seen by the coronal. The contribution lies in this componentwise use of irreducible main factors, rather than in the spectral objects themselves.

A related but distinct rigidity problem is treated in \cite{Zhu2026}. There the adjacency spectrum together with the total-walk sequence is used to determine graphs within the class of trees, and every tree of matching number at most $4$ is determined in that setting. Here a generalized cospectral mate of a tree need not itself be a tree, so its component structure and the realizability of those components must first be controlled. No result from \cite{Zhu2026} is needed for the double-star application below; the
required rigidity within the class of trees is proved directly in Section~\ref{sec:applications}.

We apply the method to the double stars
\(
T_m:=D(2m,m+1),
 m\geq2,
\)
where $D(a,b)$ is obtained by joining the centers of $K_{1,a}$ and $K_{1,b}$. The nonzero quartic factor of $\phi_{T_m}$ splits as
\[
x^4-(3m+2)x^2+2m(m+1)
=(x^2-m)(x^2-2m-2).
\]
If neither $m$ nor $2m+2$ is a perfect square, the two quadratic factors are irreducible over $\mathbb Q$, and $\kappa_{\mathrm m}(T_m)=2$. The corresponding canonical
factor-coronal decomposition fails a zeroth-moment integrality condition, excluding disconnected generalized cospectral mates. A direct adjacency-spectral argument then determines $T_m$ within the class of trees. Hence $T_m$ is DGS.

Barranca and Barrus showed that $D(2m,m+1)$ is not DS for every $m\geq2$ \cite{BarrancaBarrus2025}. Therefore $D(2m,m+1)$ is DGS but not DS whenever neither $m$ nor $2m+2$ is a perfect square. Taking $m=4t+2$ yields the explicit infinite family \(D(8t+4,4t+3), t\geq0.\)

This application also separates the present method from that of \cite{LinEtAl2026}. For the double-star family considered there, the relevant nonzero quartic characteristic factor is irreducible over $\mathbb Q$, and the proof exploits that irreducibility. For $T_m=D(2m,m+1)$ the corresponding quartic factor is reducible. Its two quadratic irreducible factors are precisely the factors retained by the canonical factor-coronal decomposition. The irreducibility argument of \cite{LinEtAl2026} is therefore unavailable; here the factorization itself is used componentwise to exclude disconnected generalized cospectral mates.

The remainder of the paper is organized as follows. Section~\ref{sec:preliminaries} collects the main spectral and coronal preliminaries. Section~\ref{sec:main-factor-components} introduces the main-factor multiplicity, proves the component bound, and derives its cycle-rank and tree-forcing consequences for generalized cospectral mates of trees. Section~\ref{sec:two-factor} treats the case $\kappa_{\mathrm m}(T)=2$, develops factor-coronal realizability obstructions, derives matching and nullity restrictions for disconnected mates, and establishes a finite matching criterion that certifies generalized cospectrality from component-coronal realizations. Section~\ref{sec:applications} applies the method to double stars and proves the DGS-but-not-DS family above.

\section{Main eigenvalues, coronals, and generalized cospectrality}\label{sec:preliminaries}

Throughout the paper, all graphs are finite, simple, and undirected. For a graph $G$ on $n$ vertices, let $A(G)$ denote its adjacency matrix and let
\(\phi_G(x):=\det(xI-A(G))\) be its characteristic polynomial. We write $\mathbf{1}$ for the all-one vector of the appropriate dimension, and $\mathbf{1}_X$ when the underlying graph $X$ needs to be specified.

For a graph $X$, let $c(X)$ denote the number of its connected components, and define its cycle rank by \(\beta(X):=|E(X)|-|V(X)|+c(X).\) For $k\ge0$, let
\(W_k(X):=\mathbf{1}_X^{T}A(X)^k\mathbf{1}_X\) denote the total number of walks of length $k$ in $X$. For a connected graph $X$, let
\(
\rho(X):=\max\{|\lambda|:\lambda
    \text{ is an adjacency eigenvalue of }X\}
\)
denote the spectral radius of $X$. By the Perron--Frobenius theorem \cite{BrouwerHaemers2012}, $\rho(X)$ is itself an adjacency eigenvalue of $X$, is simple, and has an eigenvector with all entries positive. We refer to $\rho(X)$ as the \emph{Perron root} of $X$. In particular, $\rho(X)$ is the largest adjacency eigenvalue of $X$.

For an adjacency eigenvalue $\lambda$ of $G$, let $E_\lambda$ denote the orthogonal projection onto its eigenspace. We call $\lambda$ a \emph{main eigenvalue} of $G$ if \(E_\lambda\mathbf{1}\ne0.\) In particular, if $G$ is connected, then its Perron root is main, since a positive Perron eigenvector has nonzero inner product with $\mathbf{1}$.
If $\lambda_1,\ldots,\lambda_s$ are the distinct main eigenvalues of $G$, we define the \emph{main polynomial} of $G$ by \[M_G(x):=\prod_{i=1}^{s}(x-\lambda_i).\]
The \emph{coronal} of $G$ is the rational function
\(\Gamma_G(x):=
    \mathbf{1}^{T}(xI-A(G))^{-1}\mathbf{1}.\)

\begin{lemma}[Main polynomial and the coronal]\label{lem:coronal-denominator}
Let $G$ be a graph. Then
\[
\Gamma_G(x)=\sum_{\lambda}
    \frac{\|E_\lambda\mathbf1\|^2}{x-\lambda},
\]
where the sum runs over the distinct adjacency eigenvalues of $G$. The poles of $\Gamma_G$ are precisely the main eigenvalues of $G$, and every pole is simple. If $\Gamma_G$ is written in reduced form with monic denominator, then this denominator is $M_G(x)$. Moreover, $M_G(x)\in\mathbb Z[x]$, $M_G(x)\mid\phi_G(x)$ in $\mathbb Z[x]$.
The set of main eigenvalues is closed under algebraic conjugation over $\mathbb Q$.
\end{lemma}

\begin{proof}
By the spectral decomposition of $A$,
\[
(xI-A)^{-1}=\sum_{\lambda}
    \frac{E_\lambda}{x-\lambda},
\]
where $\lambda$ ranges over the distinct adjacency eigenvalues. Thus
\[
\Gamma_G(x)=\mathbf{1}^{T}(xI-A)^{-1}\mathbf{1}
    =\sum_{\lambda\in\operatorname{Spec}(G)}
\frac{\mathbf{1}^{T}E_\lambda\mathbf{1}}{x-\lambda}
 =\sum_{\lambda\in\operatorname{Spec}(G)}
    \frac{\|E_\lambda\mathbf{1}\|^2}{x-\lambda}.
\]
The residue at $\lambda$ is nonzero exactly when $E_\lambda\mathbf1\neq0$. Hence the poles are precisely the main eigenvalues, and each is simple.

The adjugate formula gives
\[
\Gamma_G(x)=
    \frac{\mathbf1^T\operatorname{adj}(xI-A)\mathbf1}{\phi_G(x)}.
\]
Hence $\Gamma_G\in\mathbb Q(x)$, and its reduced denominator divides $\phi_G$ in $\mathbb Q[x]$. Since its poles are the distinct main eigenvalues and all are simple, its monic reduced denominator is
\(
\prod_{\lambda\ {\rm main}}(x-\lambda)=M_G(x).
\)
Since $\phi_G$ is monic and integral, Gauss' lemma gives $M_G(x)\in\mathbb Z[x]$, $M_G(x)\mid\phi_G(x)$ in $\mathbb Z[x]$. Since $M_G\in\mathbb Q[x]$, the set of its roots, namely the main eigenvalues of $G$, is closed under algebraic conjugation over $\mathbb Q$.
\end{proof}

McLeman and McNicholas~\cite[Theorem~12]{McLemanMcNicholas2011} proved the following complement--coronal identity. We record it in the form used below.

\begin{lemma}[Complement--coronal identity]\label{lem:complement-coronal}
Let $G$ be a graph on $n$ vertices. Then
\[
\phi_{\overline G}(x)=(-1)^n\phi_G(-x-1)
\bigl(1+\Gamma_G(-x-1)\bigr).
\]
\end{lemma}

\begin{corollary}[Generalized cospectral invariance]\label{cor:gs-coronal}
Let $G$ and $H$ be generalized cospectral graphs. Then \(\Gamma_G(x)=\Gamma_H(x), M_G(x)=M_H(x).\)
\end{corollary}

\begin{proof}
Let $n$ be the common order of $G$ and $H$. Generalized cospectrality gives \(\phi_G=\phi_H, \phi_{\overline G}=\phi_{\overline H}.\) Lemma~\ref{lem:complement-coronal} then yields \(\Gamma_G(x)=\Gamma_H(x).\) Lemma~\ref{lem:coronal-denominator} identifies the reduced monic denominator of the coronal with the main polynomial, and hence           \(M_G(x)=M_H(x).\)
\end{proof}

\section{Main-factor multiplicities and tree mates}\label{sec:main-factor-components}

\medskip
\noindent\textit{Main-factor allocation and component bound.}
\medskip

Write
\begin{equation}\label{eq:characteristic-irreducible-factorization}
\phi_G(x) =\prod_{i=1}^{r} f_i(x)^{m_i},
\end{equation}
where $f_1,\ldots,f_r$ are distinct monic polynomials irreducible over $\mathbb Q$.

By Lemma~\ref{lem:coronal-denominator}, $M_G$ is a square-free divisor of $\phi_G$. Thus there is a unique set \(I_{\mathrm m}(G)\subseteq\{1,\ldots,r\}\) such that
\(M_G(x)=\prod_{i\in I_{\mathrm m}(G)} f_i(x).\) The \emph{main-factor multiplicity} of $G$ is \(\kappa_{\mathrm m}(G):=\sum_{i\in I_{\mathrm m}(G)} m_i.\)

\begin{theorem}[Main-factor allocation and component bound]\label{thm:main-factor-component}
Let $G$ and $H$ be graphs with \(\phi_H=\phi_G, M_H=M_G,\) where $\phi_G$ is factorized as in \eqref{eq:characteristic-irreducible-factorization}. Let
\( H=H_1\cup\cdots\cup H_{c(H)}\) be the decomposition of $H$ into connected components, and set
\(
a_{ji}:=\operatorname{ord}_{f_i}\bigl(\phi_{H_j}\bigr),
1\le j\le c(H), 1\le i\le r.
\)
Then
\begin{equation}\label{eq:allocation-column-sum}
\sum_{j=1}^{c(H)} a_{ji}=m_i, 1\le i\le r,
\end{equation}
and every component contains at least one main irreducible factor:
\begin{equation}\label{eq:allocation-row}
\sum_{i\in I_{\mathrm m}(G)} a_{ji}\ge1, 1\le j\le c(H).
\end{equation}
Consequently, \(c(H)\le\kappa_{\mathrm m}(G).\)
\end{theorem}

\begin{proof}
The characteristic polynomial is multiplicative over disjoint unions, so \(\phi_H=\prod_{j=1}^{c(H)}\phi_{H_j}.\) Taking the multiplicity of $f_i$ on both sides and using
$\phi_H=\phi_G$ gives \eqref{eq:allocation-column-sum}.

For the row condition, fix $j$ and let $\rho_j:=\rho(H_j)$. Since $H_j$ is connected, it has a positive Perron eigenvector $v_j$. Extend $v_j$ by zero outside $V(H_j)$ to a vector $\widetilde v_j$ on $V(H)$. Then \(A(H)\widetilde v_j=\rho_j\widetilde v_j, \mathbf1^T\widetilde v_j>0.\) Thus $\rho_j$ is a main eigenvalue of $H$.

Because $M_H=M_G$, the eigenvalue $\rho_j$ lies in a factor $f_{i(j)}$ with $i(j)\in I_{\mathrm m}(G)$. Since $\rho_j$ is an eigenvalue of $H_j$, its minimal polynomial $f_{i(j)}$ divides $\phi_{H_j}$. Hence $a_{j,i(j)}\ge1$, which gives \eqref{eq:allocation-row}.

Summing \eqref{eq:allocation-row} over the components and using \eqref{eq:allocation-column-sum},
\[c(H) \le \sum_{j=1}^{c(H)} \sum_{i\in I_{\mathrm m}(G)} a_{ji}=\sum_{i\in I_{\mathrm m}(G)} m_i=\kappa_{\mathrm m}(G).\]
\end{proof}

If $H$ is generalized cospectral with $G$, then Corollary~\ref{cor:gs-coronal} and Theorem~\ref{thm:main-factor-component} give \(c(H)\le \kappa_{\mathrm m}(G).\)
In particular, if $\kappa_{\mathrm m}(G)=1$, then every generalized cospectral mate of $G$ is connected.

\begin{remark}\label{rem:allocation-principle}
Theorem~\ref{thm:main-factor-component} is useful not only through the numerical bound on $c(H)$. Its proof gives an allocation principle: every connected component of $H$ contains the minimal polynomial of its Perron root, and this polynomial is one of the irreducible factors contributing to the main polynomial. Thus the multiplicities of the
main irreducible factors form a finite supply that must be distributed among the components of $H$. The bound $c(H)\leq\kappa_{\mathrm m}(G)$ is the immediate numerical consequence of this allocation.
\end{remark}

\medskip
\noindent\textit{Generalized cospectral mates of trees.}
\medskip

For a graph $X$, let \(\beta(X):=|E(X)|-|V(X)|+c(X)\) denote its cycle rank. If \(X=X_1\cup\cdots\cup X_{c(X)}\) is its decomposition into connected components, then
\(\beta(X)=\sum_{j=1}^{c(X)}\beta(X_j),\) and $\beta(X_j)=0$ exactly when $X_j$ is a tree.
 
\begin{proposition}[Structure and cycle rank of tree mates]\label{prop:tree-mate-cycle-rank}
Let $T$ be a tree on $n$ vertices, and let $H$ be generalized cospectral with $T$. Then \(|V(H)|=n, |E(H)|=n-1,\) and $H$ is bipartite. Moreover,
\begin{equation}\label{eq:tree-mate-cycle-rank}
\beta(H)=c(H)-1 \le
    \kappa_{\mathrm m}(T)-1.
\end{equation}
If $q(H)$ denotes the number of connected components of $H$ that contain a cycle, then
\begin{equation}\label{eq:cyclic-component-bound}
q(H)\le c(H)-1 \le \kappa_{\mathrm m}(T)-1.
\end{equation}
In particular, $H$ has at least one tree component. If $\kappa_{\mathrm m}(T)=1$, then $H$ is connected and hence is a tree.
\end{proposition}

\begin{proof}
Adjacency cospectrality gives $|V(H)|=|V(T)|=n$. Since \(\operatorname{tr}A(X)^2=2|E(X)|\) for every simple graph $X$, it also gives \(|E(H)|=|E(T)|=n-1.\)

Since $T$ is bipartite,
\(
\operatorname{tr}A(T)^{2k+1}=0,
     k\ge0.
\)
The same odd spectral moments vanish for $H$. An odd cycle of length $2k+1$ in $H$ would give \(\operatorname{tr}A(H)^{2k+1}>0,\) a contradiction. Hence $H$ is bipartite.

By the definition of cycle rank, \(\beta(H)=|E(H)|-|V(H)|+c(H)=c(H)-1.\) Corollary~\ref{cor:gs-coronal} and Theorem~\ref{thm:main-factor-component} give \(c(H)\le\kappa_{\mathrm m}(T),\) and hence
\(
\beta(H)=c(H)-1
    \le
\kappa_{\mathrm m}(T)-1.
\)

Write \(H=H_1\cup\cdots\cup H_{c(H)}.\) Every component containing a cycle has cycle rank at least one, so
\[
q(H) \le \sum_{j=1}^{c(H)}\beta(H_j)=\beta(H)=c(H)-1.
\]
Together with $c(H)\le\kappa_{\mathrm m}(T)$, this gives \eqref{eq:cyclic-component-bound}. Hence at least one component of $H$ is a tree.

If $\kappa_{\mathrm m}(T)=1$, then $c(H)=1$. Hence $\beta(H)=0$, so $H$ is a tree.
\end{proof}

\begin{lemma}[Perron factor and the one-factor case]\label{lem:perron-main-factor}
Let $G$ be a connected graph, set $\rho=\rho(G)$, and let $p_\rho(x)$ be the minimal polynomial of $\rho$ over $\mathbb Q$. Then \(\operatorname{ord}_{p_\rho}\phi_G=1.\)
Moreover,
\(\kappa_{\mathrm m}(G)=1 \Longleftrightarrow  M_G(x)\ \text{is irreducible over }\mathbb Q.\)
\end{lemma}

\begin{proof}
The Perron root $\rho$ is simple and has a positive eigenvector, so it is main. Thus $p_\rho\mid M_G$ and $p_\rho\mid\phi_G$. If $p_\rho^2$ divided $\phi_G$, then $\rho$ would be a multiple root of $\phi_G$, contrary to the simplicity of the Perron eigenvalue. Hence \(\operatorname{ord}_{p_\rho}\phi_G=1.\)

If $\kappa_{\mathrm m}(G)=1$, then \(\sum_{i\in I_{\mathrm m}(G)}m_i=1.\) Thus $I_{\mathrm m}(G)$ consists of a single index, and $M_G$ is irreducible.

Conversely, if $M_G$ is irreducible, then $p_\rho\mid M_G$ gives $M_G=p_\rho$. By \(\operatorname{ord}_{p_\rho}\phi_G=1\), this factor occurs with exponent one in $\phi_G$, so
\(\kappa_{\mathrm m}(G)=1.\)
\end{proof}

\begin{corollary}[Tree forcing by an irreducible main polynomial]\label{cor:irreducible-main-tree-forcing}
Let $T$ be a tree. If $M_T(x)$ is irreducible over $\mathbb Q$, then every graph generalized cospectral with $T$ is a tree.
\end{corollary}

\begin{proof}
Lemma~\ref{lem:perron-main-factor} gives $\kappa_{\mathrm m}(T)=1$, and Proposition~\ref{prop:tree-mate-cycle-rank} gives the claim.
\end{proof}

\section{The two-factor case}\label{sec:two-factor}
Throughout this section, let $T$ be a tree with $\kappa_{\mathrm m}(T)=2$. If $H$ is generalized cospectral with $T$, then Proposition~\ref{prop:tree-mate-cycle-rank} gives \(c(H)\leq 2, \beta(H)=c(H)-1,\) and $H$ is bipartite. If $c(H)=1$, then $\beta(H)=0$, so $H$ is a tree. Otherwise, \(c(H)=2, \beta(H)=1.\)

\medskip
\noindent\textit{The two-factor structure.}
\medskip

By Lemma~\ref{lem:perron-main-factor}, the minimal polynomial of $\rho(T)$ is a main irreducible factor of $\phi_T$ occurring with multiplicity one. Since $\kappa_{\mathrm m}(T)=2$, there is exactly one further main irreducible factor, also with multiplicity one. Hence \(M_T(x)=f(x)g(x),\) where $f$ and $g$ are distinct monic polynomials irreducible over $\mathbb Q$, each occurring with exponent one in $\phi_T$.

\begin{theorem}[Two-factor structure and allocation]\label{thm:two-factor-structure}
Let $T$ be a tree with $\kappa_{\mathrm m}(T)=2$. Then \(M_T(x)=f(x)g(x),\) where $f,g\in\mathbb Z[x]$ are distinct monic polynomials irreducible over $\mathbb Q$, and
\(\operatorname{ord}_{f}\phi_T =\operatorname{ord}_{g}\phi_T=1.\) If $H$ is generalized cospectral with $T$, then either $H$ is a tree, or \(H=R\cup U,\) where $R$ is a tree and $U$ is a connected bipartite unicyclic graph. In the latter case, after interchanging $f$ and $g$ if necessary, \(M_R(x)=f(x), M_U(x)=g(x).\)
\end{theorem}

\begin{proof}
Since $\kappa_{\mathrm m}(T)=2$, Lemma~\ref{lem:perron-main-factor} implies that $M_T$ is reducible. Hence $M_T$ has at least two irreducible factors. On the other hand,
\(\sum_{i\in I_{\mathrm m}(T)}m_i=2,\) and each $m_i$ is positive. Thus there are exactly two main irreducible factors, each occurring with multiplicity one in
$\phi_T$. Thus $M_T=fg$, and both factors occur with multiplicity one in $\phi_T$.

Let $H$ be generalized cospectral with $T$. Proposition~\ref{prop:tree-mate-cycle-rank} gives
\(
c(H)\le2, 
\beta(H)=c(H)-1,
\)
and $H$ is bipartite. If $c(H)=1$, then $H$ is a tree.

Suppose that $c(H)=2$. Then $\beta(H)=1$. By additivity of the cycle rank, one component has cycle rank zero and the other has cycle rank one. After relabeling, $H=R\cup U$,
where $R$ is a tree and $U$ is connected and unicyclic. Since $H$ is bipartite, so is $U$.

By Theorem~\ref{thm:main-factor-component}, each of $R$ and $U$ contains at least one of the two main factors $f$ and $g$. Since both factors occur with multiplicity one in $\phi_T$, there are only two main-factor copies altogether. Hence each component contains exactly one of them. After interchanging $f$ and $g$ if necessary,
\(\operatorname{ord}_f\phi_R=1, \operatorname{ord}_g\phi_U=1,\) while \(\operatorname{ord}_g\phi_R=\operatorname{ord}_f\phi_U=0.\)

A main eigenvalue of a component remains main in the disjoint union after extending a corresponding eigenvector by zero. Since \(M_H=M_T=fg,\) every irreducible factor of $M_R$ or $M_U$ is one of $f$ and $g$. Both $R$ and $U$ are connected, so each has a nonconstant main polynomial. The allocation above leaves only $f$ as a possible main factor of $R$ and only $g$ as a possible main factor of $U$. Hence \( M_R=f, M_U=g.\)
\end{proof}

\medskip
\noindent\textit{The factor-coronal decomposition.}
\medskip

By Lemma~\ref{lem:coronal-denominator}, the reduced denominator of $\Gamma_T$ is $M_T=fg$. Write
\[
\Gamma_T(x)=
\frac{P(x)}{f(x)g(x)},
\]
where
\(
P\in\mathbb Q[x],
\deg P<\deg f+\deg g,
\gcd(P,fg)=1.
\)

\begin{theorem}[Canonical factor-coronal decomposition]\label{thm:factor-coronal-decomposition}
There exist unique polynomials $A,B\in\mathbb Q[x]$ such that \(P=Ag+Bf, \deg A<\deg f, \deg B<\deg g.\) Define 
\(\Gamma_f(x):=\frac{A(x)}{f(x)}, \Gamma_g(x):=\frac{B(x)}{g(x)}.\)
Then $\Gamma_T=\Gamma_f+\Gamma_g$, and both fractions are reduced. If $H=R\cup U$ is a disconnected generalized cospectral mate of $T$ with \(M_R=f, M_U=g,\) then \(\Gamma_R=\Gamma_f, \Gamma_U=\Gamma_g.\)
\end{theorem}

\begin{proof}
Since $f$ and $g$ are coprime, $\frac{P}{fg}$ has a unique proper partial-fraction decomposition
\[
\frac{P}{fg}=\frac{A}{f}+\frac{B}{g},
\]
with \(\deg A<\deg f, \deg B<\deg g.\) Equivalently, $P=Ag+Bf$.

Since $\gcd(P,fg)=1$, \(\gcd(P,f)=\gcd(P,g)=1.\) From $P=Ag+Bf$ and $\gcd(f,g)=1$, reduction modulo $f$ and modulo $g$ gives \(\gcd(A,f)=1, \gcd(B,g)=1.\) Thus $\Gamma_f$ and $\Gamma_g$ are reduced.

Let $H=R\cup U$ be as in the statement. Since $xI-A(H)$ is block diagonal with blocks $xI-A(R)$ and $xI-A(U)$, \(\Gamma_H=\Gamma_R+\Gamma_U.\)
Corollary~\ref{cor:gs-coronal} gives \(\Gamma_H=\Gamma_T.\) Since $M_R=f$ and $M_U=g$, Lemma~\ref{lem:coronal-denominator} gives reduced proper forms
\(\Gamma_R=\frac{A_R}{f}, \Gamma_U=\frac{B_U}{g},\) with \(\deg A_R<\deg f, \deg B_U<\deg g.\)
Hence
\(
\frac{A_R}{f}+\frac{B_U}{g}=\Gamma_T=\frac{A}{f}+\frac{B}{g}.
\)
Uniqueness of the proper partial-fraction decomposition gives \(A_R=A, B_U=B.\) Therefore
\(
\Gamma_R=\Gamma_f, \Gamma_U=\Gamma_g.
\)
\end{proof}

If a disconnected generalized cospectral mate exists, $\Gamma_f$ and $\Gamma_g$ must be realizable as coronals of connected graphs.

\subsection{Factor-moment realizability obstructions}\label{subsec:factor-moments}
For $h\in\{f,g\}$, write the Laurent expansion at infinity as
\begin{equation}\label{eq:factor-moment-expansion}
\Gamma_h(x)=\sum_{k\ge0} w_k^{(h)}x^{-k-1}.
\end{equation}
The coefficient $w_k^{(h)}$ is called the $k$th \emph{factor moment} associated with $h$.

If $\Gamma_h=\Gamma_X$ for a graph $X$, expansion at infinity gives
\[
\Gamma_h(x)=\Gamma_X(x)=\sum_{k\ge0}
    \mathbf1^TA(X)^k\mathbf1\,x^{-k-1}.
\]
Hence
\(
w_k^{(h)}=\mathbf1^TA(X)^k\mathbf1=W_k(X),
 k\ge0.
\)
In particular, \(w_k^{(h)}\in\mathbb Z_{\ge0}, k\ge0,\) and
\begin{equation}\label{eq:first-three-walk-moments}
w_0^{(h)}=|V(X)|,\quad
w_1^{(h)}=2|E(X)|,\quad
w_2^{(h)}=\sum_{v\in V(X)}d_X(v)^2.
\end{equation}

Let \(h(x)=x^d+c_{d-1}x^{d-1}+\cdots+c_1x+c_0.\) Since $h(x)\Gamma_h(x)$ is a polynomial, comparison of the coefficient of $x^{-k-1}$ gives
\begin{equation}\label{eq:factor-moment-recurrence}
w_{k+d}^{(h)}+c_{d-1}w_{k+d-1}^{(h)}+\cdots+c_1w_{k+1}^{(h)}+c_0w_k^{(h)}=0,
\quad k\ge0.
\end{equation}
Since $h\in\mathbb Z[x]$, recurrence \eqref{eq:factor-moment-recurrence} determines every later term integrally from the preceding $d$ terms. Hence the factor-moment
sequence is integral if and only if
\(
w_0^{(h)},\ldots,w_{d-1}^{(h)}\in\mathbb Z.
\)
The recurrence alone does not reduce nonnegativity to a finite check.

\begin{theorem}[Factor-moment realizability obstruction]\label{thm:unified-factor-coronal-obstruction}
Let $T$ be a tree with $\kappa_{\mathrm m}(T)=2$ and $M_T=fg$. Let the factor moments be defined by \eqref{eq:factor-moment-expansion}. If $T$ has a disconnected
generalized cospectral mate, then there is an ordering \((p,q)\in\{(f,g),(g,f)\}\) such that $p$ corresponds to the tree component and $q$ to the connected bipartite unicyclic component. Set \(n_p:=w_0^{(p)}, n_q:=w_0^{(q)}.\) Then:

\medskip
\noindent
{\rm (i)}
\(w_k^{(p)},\,w_k^{(q)}\in\mathbb Z_{\ge0}, k\ge0,\) with \(n_p\ge1, n_q\ge4.\)

\medskip
\noindent
{\rm (ii)}
For \(\delta_h:=w_1^{(h)}-2w_0^{(h)},\) one has \(\delta_p=-2, \delta_q=0.\)

\medskip
\noindent
{\rm (iii)}
The second moments are even and satisfy \(w_2^{(p)} \ge \max\{0,4n_p-6\},  w_2^{(q)} \ge 4n_q.\) If $n_p\ge2$, equality in \(w_2^{(p)}\ge4n_p-6\) holds exactly when the tree component is a path. Equality in \(w_2^{(q)}\ge4n_q\) holds exactly when the unicyclic component is a cycle.

\medskip
\noindent
{\rm (iv)}
If $p(x)=x-a$, then \( (a,n_p)\in\{(0,1),(1,2)\},\) and the corresponding tree is $K_1$ or $K_2$, respectively. If $q(x)=x-a$, then $a=2$ and the corresponding component is
the even cycle $C_{n_q}$.
\end{theorem}

\begin{proof}
By Theorems~\ref{thm:two-factor-structure} and \ref{thm:factor-coronal-decomposition}, a disconnected mate has the form $H=R\cup U$, and the ordering $(p,q)$ may be chosen so that \(\Gamma_p=\Gamma_R, \Gamma_q=\Gamma_U,\) where $R$ is a tree and $U$ is connected, bipartite, and unicyclic.

\noindent{\rm (i)}
Comparison of Laurent expansions gives \(w_k^{(p)}=W_k(R),\quad w_k^{(q)}=W_k(U), \quad k\ge0.\) Hence \(w_k^{(p)},\,w_k^{(q)}\in\mathbb Z_{\ge0}.\) Also $n_p=|V(R)|\ge1$.
The unique cycle of $U$ is even and has length at least four, so $n_q=|V(U)|\ge4$.

\noindent{\rm (ii)}
Since \(|E(R)|=n_p-1, |E(U)|=n_q,\) we have \(w_1^{(p)}=2n_p-2, w_1^{(q)}=2n_q.\) Thus \(\delta_p=-2, \delta_q=0.\)

\noindent{\rm (iii)}
If $n_p=1$, then $R=K_1$ and $w_2^{(p)}=0$. For $n_p\ge2$,
\[
w_2^{(p)}=\sum_{v\in V(R)}d_R(v)^2 =4n_p-8+\sum_{v\in V(R)}(d_R(v)-2)^2 \ge 4n_p-6.
\]
Here $\sum_{v\in V(R)}(d_R(v)-2)=-2$ and $d_R(v)\ge1$. Equality holds exactly when two vertices have degree one and all others have degree two, namely when $R$ is a path.
For $U$, \(w_2^{(q)}=4n_q+\sum_{v\in V(U)}(d_U(v)-2)^2\ge 4n_q,\) with equality exactly when $U$ is a cycle. Moreover, \(\sum_v d(v)^2 \equiv \sum_v d(v) \equiv0\pmod2,\)
so both second moments are even.

\noindent{\rm (iv)}
If $p(x)=x-a$, then $M_R(x)=x-a$. Thus $a$ is the only main eigenvalue of $R$, and $A(R)\mathbf1=a\mathbf1$. Hence $R$ is $a$-regular. The only connected regular trees are
$K_1$ and $K_2$, giving \((a,n_p)=(0,1) \quad \text{or} \quad (1,2).\) If $q(x)=x-a$, the same argument makes $U$ $a$-regular. Since $U$ is unicyclic, its average degree is two, so $a=2$ and $U$ is a cycle. Its bipartiteness makes the cycle even.
\end{proof}

For the linear cases,
\[
\begin{alignedat}{3}
p(x)=x
    &\quad\Longrightarrow\quad&
  R&=K_1,
   &\qquad
 \Gamma_R(x)&=\frac{1}{x},\\
    p(x)=x-1
    &\quad\Longrightarrow\quad&
    R&=K_2,
    &
    \Gamma_R(x)&=\frac{2}{x-1},\\[2mm]
    q(x)=x-2
    &\quad\Longrightarrow\quad&
    U&=C_{n_q},
    &
  \Gamma_U(x)&=\frac{n_q}{x-2},
\end{alignedat}
\]
where $n_q$ is even and $n_q\ge4$. Thus a linear main factor $x-a$ with $a\notin\{0,1,2\}$ rules out a disconnected generalized cospectral mate.

\medskip
\noindent\textit{A factor-coronal tree-forcing obstruction.}
\medskip

\begin{corollary}[Factor-coronal tree forcing]\label{cor:unified-tree-forcing}
Let $T$ be a tree with \(\kappa_{\mathrm m}(T)=2, M_T=fg,\) and let \(\Gamma_T=\Gamma_f+\Gamma_g\) be the canonical decomposition of Theorem~\ref{thm:factor-coronal-decomposition}. If neither ordering $(p,q)=(f,g)$ nor $(p,q)=(g,f)$ satisfies conditions {\rm (i)}--{\rm (iv)} of Theorem~\ref{thm:unified-factor-coronal-obstruction}, then every graph generalized cospectral with $T$ is a tree.
\end{corollary}

\begin{proof}
If a disconnected generalized cospectral mate existed, Theorem~\ref{thm:unified-factor-coronal-obstruction} would give an ordering of $f$ and $g$ satisfying conditions
{\rm (i)}--{\rm (iv)}. Hence no disconnected mate exists.

Every generalized cospectral mate of $T$ is therefore connected, and Proposition~\ref{prop:tree-mate-cycle-rank} gives that it is a tree.
\end{proof}

A single nonintegral or negative factor moment certifies nonrealizability. No finite test for nonnegativity of the full factor-moment sequence is implied.

\subsection{Matching restrictions and finite certification}\label{subsec:matching-restrictions}
The factor-moment conditions above are componentwise. We now develop complementary restrictions arising from the matching polynomial and the nullity of a possible disconnected mate. Besides yielding additional structural constraints, these relations also provide a finite certification mechanism: once the canonical factor coronals are realized by candidate graph components, finitely many matching identities can, under a degree condition, force full generalized cospectrality.

For a graph $X$, let $m_k(X)$ denote the number of $k$-matchings of $X$, with $m_0(X)=1$, and define \[\mu_X(x):=\sum_{k\ge0}(-1)^k m_k(X)x^{|V(X)|-2k}.\] The matching polynomial is multiplicative over disjoint unions, and for every forest $X$, $\phi_X(x)=\mu_X(x)$ \cite{BrouwerHaemers2012}.

We first derive necessary conditions for an actual disconnected generalized cospectral mate. Suppose that $T$ has a disconnected generalized cospectral mate
\(H=R\cup U,\) where $R$ is a tree and $U$ is connected, bipartite, and unicyclic, as in Theorem~\ref{thm:two-factor-structure}. Let $C=C_{2\ell}$, $\ell\geq 2$, be the unique cycle of $U$, and put \(F:=U-V(C),\) which is a forest.

\begin{theorem}[Matching relations for a disconnected mate]\label{thm:matching-coefficient-relation}
With the notation above, \(\phi_U(x)=\mu_U(x)-2\mu_F(x).\) For every $k\ge0$, \(m_k(T)=m_k(H)- 2(-1)^\ell m_{k-\ell}(R\cup F),\) with the convention $m_j(X)=0$ for $j<0$. Hence \(m_k(T)=m_k(H),  k<\ell,\) while \(m_\ell(T)-m_\ell(H)=-2(-1)^\ell.\)
\end{theorem}

\begin{proof}
In the determinant expansion of $\phi_U(x)$, a nonzero permutation consists of fixed points, edge transpositions, and possibly one of the two orientations of the unique cycle
$C=C_{2\ell}$. The terms not containing $C$ are exactly the matching terms and contribute $\mu_U(x)$.

Each orientation of $C$ contributes \((-1)^{2\ell-1}(-1)^{2\ell}=-1.\) If it is accompanied by a $j$-matching of $F=U-V(C)$, the additional contribution is $(-1)^j$, with $|V(F)|-2j$ fixed points. The terms containing $C$ therefore contribute
\[
-2\sum_{j\ge0}(-1)^j m_j(F)x^{|V(F)|-2j}=-2\mu_F(x).
\]
This gives \(\phi_U(x)=\mu_U(x)-2\mu_F(x).\) Cospectrality and $H=R\cup U$ give \(\phi_T=\phi_H=\phi_R\phi_U.\) Since $T$ and $R$ are trees, \(\mu_T=\mu_R(\mu_U-2\mu_F)=\mu_H-2\mu_{R\cup F}.\)

Let $n=|V(T)|$. Since $|V(R\cup F)|=n-2\ell$, comparison of the coefficients of $x^{n-2k}$ gives
\[
(-1)^k m_k(T)=(-1)^k m_k(H)- 2(-1)^{k-\ell}m_{k-\ell}(R\cup F).
\]
Hence \(m_k(T)=m_k(H)-2(-1)^\ell m_{k-\ell}(R\cup F).\)
For $k<\ell$ the last matching number is zero, whereas for $k=\ell$ it equals $m_0(R\cup F)=1$.
The two final assertions follow.
\end{proof}

As an immediate consequence of Theorem \ref{thm:matching-coefficient-relation}, if $|C|\equiv0\pmod4$, then
\(
m_\ell(T)=m_\ell(H)-2,
\)
whereas if $|C|\equiv2\pmod4$, then
\(
m_\ell(T)=m_\ell(H)+2.
\)
Thus the number of $\ell$-matchings need not be determined by the adjacency spectrum.

\begin{corollary}[Cycle-length restriction]\label{cor:Cycle-length restriction}
With the notation of Theorem~\ref{thm:matching-coefficient-relation}, the unique cycle $C=C_{2\ell}$ of the unicyclic component $U$ has length at least $6$. Equivalently,
$\ell \geq 3$.
\end{corollary}

\begin{proof}
For any simple graph $X$ with $e=|E(X)|$, the number of
$2$-matchings satisfies
\[
m_2(X)=
\binom{e}{2}
-
\sum_{v\in V(X)} \binom{d_X(v)}{2}.
\]
Since
\[
W_2(X)=
\mathbf{1}^{\mathsf T}A(X)^2\mathbf{1}=
\sum_{v\in V(X)} d_X(v)^2
\]
and
\[
\sum_{v\in V(X)} d_X(v)=2e,
\]
it follows that
\[
m_2(X)=
\frac{e^2+e-W_2(X)}{2}.
\]

Now let $H=R\cup U$ be a disconnected generalized cospectral mate of the tree $T$. Adjacency cospectrality gives \(|E(T)|=|E(H)|.\) Moreover, by Corollary~\ref{cor:gs-coronal},
\(
\Gamma_T(x)=\Gamma_H(x).
\)
Comparing the coefficient of $x^{-3}$ in the Laurent expansions at infinity gives
\(
W_2(T)=W_2(H).
\)
Hence the preceding identity yields
\(
m_2(T)=m_2(H).
\)

Suppose, to the contrary, that the unique cycle of $U$ is a $4$-cycle. Then $\ell=2$. By Theorem~\ref{thm:matching-coefficient-relation},
\[
m_\ell(T)-m_\ell(H)=-2(-1)^\ell,
\]
and therefore
\(
m_2(T)-m_2(H)=-2,
\)
contradicting $m_2(T)=m_2(H)$. Thus $\ell\neq 2$. Since $U$ is bipartite and unicyclic, $\ell\geq 2$, and consequently
\(
\ell\geq 3.
\)
Hence the unique cycle of $U$ has length at least six.
\end{proof}

For a graph $X$, let $\eta(X)$ denote the multiplicity of $0$ as an adjacency eigenvalue, and let $\nu(X)$ denote its matching number.

\begin{proposition}[Nullity and maximum-matching restrictions]\label{prop:unicyclic-nullity-restriction}
With the notation above, the following hold.

\medskip
\noindent
{\rm (i)}
The nullity of $U$ satisfies
\begin{equation}\label{eq:bipartite-unicyclic-nullity-options}
\eta(U) \in \{|V(U)|-2\nu(U),\, |V(U)|-2\nu(U)+2 \}.
\end{equation}
Moreover, \(\nu(T)-\nu(R) \in \{\nu(U),\,\nu(U)-1\}.\)

\medskip
\noindent
{\rm (ii)}
Put $s:=\nu(F)$. The exceptional case
\begin{equation}\label{eq:exceptional-nullity}
\eta(U)=|V(U)|-2\nu(U)+2
\end{equation}
occurs if and only if
\begin{equation}\label{eq:exceptional-nullity-conditions}
\nu(U)=\ell+s,\quad
\ell\equiv0\pmod2,\quad
m_{\nu(U)}(U)=2m_s(F).
\end{equation}
Equivalently,
\(\nu(T)-\nu(R)=\nu(U)-1\)
if and only if \eqref{eq:exceptional-nullity-conditions} holds. Otherwise, $\nu(T)-\nu(R)=\nu(U)$.
\end{proposition}

\begin{proof}
\noindent{\rm (i)}
Guo, Yan and Yeh \cite{GuoYanYeh2009} proved that an $n$-vertex unicyclic graph $X$ satisfies
\[
\eta(X) \in \{ n-2\nu(X)-1,\, n-2\nu(X),\, n-2\nu(X)+2\}.
\]
Since $U$ is bipartite, its adjacency matrix has the form
\[
A(U) =
\begin{pmatrix}
 0&B\\
 B^T&0
\end{pmatrix}.
\]
Thus \(\operatorname{rank}A(U)=2\operatorname{rank}B,\) and hence $\eta(U)\equiv |V(U)|\pmod2$. The value $|V(U)|-2\nu(U)-1$ is therefore impossible, giving \eqref{eq:bipartite-unicyclic-nullity-options}.

For a forest $X$, \(\eta(X)=|V(X)|-2\nu(X).\) Since \(\phi_T=\phi_R\phi_U,\) we have \(\eta(T)=\eta(R)+\eta(U).\) Together with \(|V(T)|=|V(R)|+|V(U)|,\) this gives
\begin{equation}\label{eq:nullity-matching-difference}
 \eta(U)=|V(U)|-2\bigl(\nu(T)-\nu(R)\bigr).
\end{equation}
Comparison with \eqref{eq:bipartite-unicyclic-nullity-options} yields \(\nu(T)-\nu(R) \in \{\nu(U),\,\nu(U)-1\}.\)

\medskip
\noindent{\rm (ii)}
Put $s=\nu(F)$. A maximum matching of $F$ together with a perfect matching of $C_{2\ell}$ gives \(\nu(U)\ge\ell+s.\) The lowest nonzero terms of $\mu_U$ and $\mu_F$ have degrees $|V(U)|-2\nu(U)$ and $|V(U)|-2\ell-2s$, respectively.

If $\nu(U)>\ell+s$, the lowest-degree term of $\mu_U$ cannot cancel in \(\phi_U=\mu_U-2\mu_F.\) Hence \(\eta(U)=|V(U)|-2\nu(U),\) so \eqref{eq:exceptional-nullity} requires \(\nu(U)=\ell+s.\) Under this equality, the coefficient of the common lowest power in $\phi_U$ is \((-1)^{\nu(U)}m_{\nu(U)}(U)-2(-1)^s m_s(F).\) Its vanishing is equivalent to \((-1)^\ell m_{\nu(U)}(U)=2m_s(F).\) Both matching counts are positive, so cancellation occurs exactly when \(\ell\equiv0\pmod2,  m_{\nu(U)}(U)=2m_s(F).\) This proves the necessity of \eqref{eq:exceptional-nullity-conditions}.

Conversely, assume \eqref{eq:exceptional-nullity-conditions}. The lowest-degree coefficient of $\phi_U$ vanishes, so \(\eta(U)>|V(U)|-2\nu(U).\) Equation~\eqref{eq:bipartite-unicyclic-nullity-options} then gives \(\eta(U)=|V(U)|-2\nu(U)+2.\)

Finally, \eqref{eq:nullity-matching-difference} shows that \eqref{eq:exceptional-nullity} is equivalent to \(\nu(T)-\nu(R)=\nu(U)-1.\) In the other nullity case,
\(\nu(T)-\nu(R)=\nu(U).\)
\end{proof}

\begin{corollary}[Matching-number gap]
With the notation above, let $C=C_{2\ell}$ be the unique cycle of the unicyclic component $U$. Then \(\nu(T)-\nu(R)\geq \ell-1\geq 2.\) In particular,\[\nu(R)\leq \nu(T)-2.\]
\end{corollary}

\begin{proof}
By Corollary~\ref{cor:Cycle-length restriction}, the unique cycle of $U$ has length at least $6$, so $\ell\geq 3$. Since $U$ contains the cycle $C_{2\ell}$, it contains a matching of size $\ell$. Hence \(\nu(U)\geq \ell.\)

By Proposition~\ref{prop:unicyclic-nullity-restriction}(i),
\[
\nu(T)-\nu(R)\in\{\nu(U),\,\nu(U)-1\}.
\]
Therefore
\(
\nu(T)-\nu(R)
\geq \nu(U)-1
\geq \ell-1
\geq 2.
\)
Equivalently,
\(
\nu(R)\leq \nu(T)-2.
\)
\end{proof}

We now reverse the viewpoint. Instead of assuming that $H=R\cup U$ is already a generalized cospectral mate, suppose only that the two canonical factor coronals are realized by a tree and a connected bipartite unicyclic graph. Equality of the coronals determines the main spectral part, but it does not a priori determine the remaining factor of the characteristic polynomial. The next result shows that, when this remaining degree is sufficiently small, only finitely many of the matching relations above are needed to recover the full characteristic polynomial.

\begin{proposition}[Finite matching certification]\label{prop:finite-matching-certification}
Let $T$ be a tree on $n$ vertices with $k_{\mathrm m}(T)=2$, write \(M_T(x)=f(x)g(x),\) and let \(\Gamma_T=\Gamma_f+\Gamma_g\) be the canonical factor-coronal decomposition of Theorem~\ref{thm:factor-coronal-decomposition}. Suppose that, after interchanging $f$ and $g$ if necessary, there exist a tree $R$ and a connected bipartite unicyclic graph $U$ such that \(\Gamma_R=\Gamma_f, \Gamma_U=\Gamma_g.\) Set \(H:=R\cup U.\) Let $C=C_{2\ell}$ be the unique cycle of $U$, and put \(F:=U-V(C).\)
Define
\[
d:=\deg\frac{\phi_T}{M_T}
   =n-\deg M_T.
\]

Let $s$ be an integer with
\(
0\leq s\leq \left\lfloor\frac{n}{2}\right\rfloor.
\)
Suppose that
\begin{equation}\label{eq:finite-matching-relations}
m_k(T)
=m_k(H)-2(-1)^\ell
m_{k-\ell}(R\cup F),
\quad
0\leq k\leq s,
\end{equation}
where $m_j(X)=0$ for $j<0$. If $d\leq 2s+1$, then $T$ and $H$ are generalized cospectral.

In particular, the conclusion holds automatically if $d\leq 3$. It also holds automatically if the unique cycle of $U$ has length at least six and $d\leq 5$.
\end{proposition}

\begin{proof}
Since \(\Gamma_H=\Gamma_R+\Gamma_U=\Gamma_f+\Gamma_g=\Gamma_T,\) comparison of the leading Laurent coefficient at infinity gives \(|V(H)|=|V(T)|=n.\)
By Lemma~\ref{lem:coronal-denominator}, the reduced monic denominator of a graph coronal is its main polynomial. Hence \(M_H=M_T.\) Since the main polynomial divides the characteristic polynomial, we may write
\(
\phi_T=M_TQ_T,
\phi_H=M_TQ_H.
\)
Consequently,
\begin{equation}\label{eq:main-divides-difference}
M_T\mid (\phi_T-\phi_H).
\end{equation}

Because $T$ is a tree, \(\phi_T(x)=\mu_T(x).\) On the other hand, $R$ is a tree and $U$ is unicyclic with unique cycle $C=C_{2\ell}$. The determinant expansion used in the proof of Theorem~\ref{thm:matching-coefficient-relation} gives, independently of any cospectrality assumption, \(\phi_U(x)=\mu_U(x)-2\mu_F(x).\) Therefore
\(
\phi_H(x)=\phi_R(x)\phi_U(x)=\mu_H(x)-2\mu_{R\cup F}(x).
\)
Thus the coefficient of $x^{n-2k}$ in $\phi_H$ is
\(
(-1)^k m_k(H)
-
2(-1)^{k-\ell}m_{k-\ell}(R\cup F).
\)
Since
\(
(-1)^{k-\ell}=(-1)^k(-1)^\ell,
\)
this coefficient can be written as
\[
(-1)^k
\left(
m_k(H)-2(-1)^\ell
m_{k-\ell}(R\cup F)
\right).
\]
By \eqref{eq:finite-matching-relations}, this agrees with the coefficient $(-1)^k m_k(T)$ of $x^{n-2k}$ in $\phi_T$ for every $0\leq k\leq s$.

Both $T$ and $H$ are bipartite, so their characteristic polynomials contain only powers having the same parity as $n$. Hence, unless $\phi_T=\phi_H$ already,
\begin{equation}\label{eq:difference-degree-bound}
\deg(\phi_T-\phi_H)
\leq n-2s-2.
\end{equation}
On the other hand, $d\leq 2s+1$ gives \(\deg M_T=n-d \geq n-2s-1.\) Therefore
\(
\deg M_T
>
\deg(\phi_T-\phi_H).
\)
Together with \eqref{eq:main-divides-difference}, this forces \(\phi_T=\phi_H.\)

We already have $\Gamma_T=\Gamma_H$. Applying the complement--coronal identity of Lemma~\ref{lem:complement-coronal} to $T$ and $H$ now gives
\(\phi_{\overline{T}}=\phi_{\overline{H}}.\) Hence $T$ and $H$ are generalized cospectral.

It remains to verify the two stated special cases. Equality $\Gamma_T=\Gamma_H$ implies equality of all total walk counts. In particular, \(W_1(T)=W_1(H),\)
and therefore \(|E(T)|=|E(H)|.\) Thus \(m_1(T)=m_1(H).\) Since $\ell\geq 2$, the correction term in \eqref{eq:finite-matching-relations} vanishes for $k=0,1$.
Hence the relations for $k=0,1$ hold automatically. Taking $s=1$ proves the conclusion whenever $d\leq 3$.

For the second special case, assume directly that the unique cycle of the candidate graph $U$ has length at least $6$. Thus $\ell\geq 3$. For any simple graph $X$ with $e=|E(X)|$,
\[
m_2(X)
=
\binom{e}{2}
-
\sum_{v\in V(X)}
\binom{d_X(v)}{2}
=
\frac{e^2+e-W_2(X)}{2}.
\]
Since $\Gamma_T=\Gamma_H$ gives both $|E(T)|=|E(H)|$ and $W_2(T)=W_2(H)$, we obtain $m_2(T)=m_2(H)$.

Because $\ell\geq 3$, the correction term also vanishes for $k=2$. Hence \eqref{eq:finite-matching-relations} holds automatically for $k=0,1,2$. Taking $s=2$ proves the conclusion whenever $d\leq 5$.
\end{proof}

Proposition~\ref{prop:finite-matching-certification} shows that the factor-coronal and matching approaches are complementary. The factor coronals determine the main spectral part and force the divisibility \(M_T\mid(\phi_T-\phi_H),\) while finitely many matching relations control the leading coefficients of the remaining characteristic-polynomial
difference. Once these two pieces overlap in degree, the entire difference must vanish.

\medskip
\noindent\textit{Sharpness of the two-factor case.}
\medskip

The condition $\kappa_{\mathrm m}(T)=2$ does not force generalized cospectral mates of $T$ to be trees. Moreover, the bound \(c(H)\le\kappa_{\mathrm m}(T)\) from Theorem~\ref{thm:main-factor-component} is attained by a generalized cospectral mate of a tree in the two-factor case.

Let $T=S_{2,2,2}$ be the subdivided claw obtained from $K_{1,3}$ by subdividing each edge once, and let $H=C_6\cup K_1$. These graphs are shown in Figure~\ref{fig:sharpness-two-factor}.

\begin{figure}[H]
\centering

\begin{minipage}{0.36\textwidth}
\centering
\color{black}
\begin{tikzpicture}[scale=0.65, every node/.style={circle, fill, inner sep=1.2pt}]
\node (c) at (0,0) {};
\node (a1) at (0,1.2) {};
\node (a2) at (0,2.4) {};
\node (b1) at (-1.05,-0.6) {};
\node (b2) at (-2.1,-1.2) {};
\node (d1) at (1.05,-0.6) {};
\node (d2) at (2.1,-1.2) {};
\draw (c)--(a1)--(a2);
\draw (c)--(b1)--(b2);
\draw (c)--(d1)--(d2);
\end{tikzpicture}

\medskip
(a) $T=S_{2,2,2}$
\end{minipage}
\hspace{0.03\textwidth}
\begin{minipage}{0.36\textwidth}
\centering
\color{black}
\begin{tikzpicture}[scale=0.65, every node/.style={circle, fill, inner sep=1.2pt}]
\foreach \i in {1,...,6}{
  \node (v\i) at ({60*(\i-1)}:1.35) {};
}
\draw (v1)--(v2)--(v3)--(v4)--(v5)--(v6)--(v1);
\node (u) at (2.7,0) {};
\end{tikzpicture}

\bigskip
(b) $H=C_6\cup K_1$
\end{minipage}

\caption{The sharpness example in the two-factor case. The tree $T=S_{2,2,2}$ and the disconnected graph $H=C_6\cup K_1$ are generalized cospectral; moreover, $c(H)=2=\kappa_{\mathrm m}(T)$.}
\label{fig:sharpness-two-factor}
\end{figure}
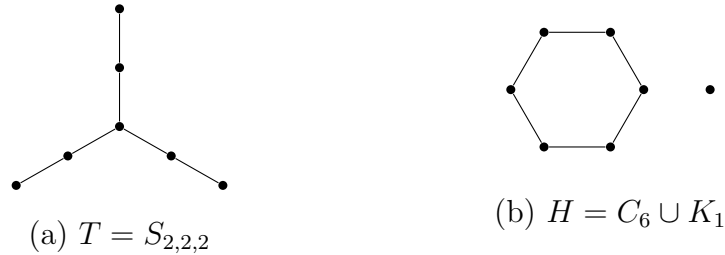

A direct calculation gives
\begin{equation}\label{eq:sharpness-char}
\phi_T(x)=x(x^2-4)(x^2-1)^2.
\end{equation}

For the coronal of $T$, symmetry gives three coordinates $r,s,t$ for \((xI-A(T))^{-1}\mathbf{1},\) corresponding respectively to the center, the three subdividing
vertices, and the three leaves. They satisfy
\[
xr-3s=1,\quad
-r+xs-t=1,\quad
-s+xt=1.
\]
Solving gives
\[
r+3s+3t
=
\frac{7x-2}{x(x-2)}.
\]
Hence
\begin{equation}
\Gamma_T(x)
=
\frac{7x-2}{x(x-2)}
=
\frac{1}{x}
+
\frac{6}{x-2}.
\label{eq:sharpness-coronal}
\end{equation}

The fraction in \eqref{eq:sharpness-coronal} is reduced, so Lemma~\ref{lem:coronal-denominator} gives \(M_T(x)=x(x-2).\) By \eqref{eq:sharpness-char}, both main factors $x$ and $x-2$ occur with multiplicity one in $\phi_T$. Therefore \(k_{\mathrm m}(T)=2.\) Moreover, by uniqueness in Theorem~\ref{thm:factor-coronal-decomposition},
\eqref{eq:sharpness-coronal} is the canonical factor-coronal decomposition of $T$.

Now set \[R:=K_1, \quad U:=C_6, \quad H:=R\cup U.\] Clearly, \(\Gamma_R(x)=\frac{1}{x}.\) Since $C_6$ is $2$-regular, \(\Gamma_U(x)=\frac{6}{x-2}.\) Thus the two canonical factor coronals of $T$ are realized exactly by the tree $R$ and the connected bipartite unicyclic graph $U$:
\(
\Gamma_R=\Gamma_x,
\Gamma_U=\Gamma_{x-2}.
\)
Here $|V(T)|=7$ and $\deg M_T=2$, so \(d=|V(T)|-\deg M_T=5.\) The unique cycle of $U=C_6$ has length six. Hence the second special case of Proposition~\ref{prop:finite-matching-certification} applies, and it follows that $T$ and $H=K_1\cup C_6$ are generalized cospectral.

Finally, \(c(H)=2=k_{\mathrm m}(T),\) so the component bound in Theorem~\ref{thm:main-factor-component} is attained.

The corresponding factor-moment sequences begin as $(1,0,0,\ldots)$ and $(6,12,24,48,\ldots)$. Their defects are $-2$ and $0$, respectively, and the unicyclic second-moment bound is attained: \(24=4\cdot 6.\) Thus the defect conditions and the second-moment bounds in Theorem~\ref{thm:unified-factor-coronal-obstruction} are simultaneously attained in this example.

The matching relations are exact here as well. With \(R=K_1, U=C_6, \ell=3, F=\varnothing,\) Theorem~\ref{thm:matching-coefficient-relation} gives
\[
m_3(T)-m_3(H)=4-2=2=-2(-1)^3,\]
while
\(
\nu(T)-\nu(R)=3=\nu(U),\)
in agreement with Proposition~~\ref{prop:unicyclic-nullity-restriction}. Thus this example realizes simultaneously the component bound, the factor-coronal decomposition, the matching restrictions, and the finite certification mechanism of Proposition~\ref{prop:finite-matching-certification}.

\section{Applications to DGS trees}\label{sec:applications}
Tree forcing alone does not imply DGS. Once every generalized cospectral mate of a tree $T$ is known to be a tree, it remains to determine $T$ within the class of trees. For the double-star family considered below, adjacency cospectrality already gives this remaining rigidity.

\medskip
\noindent\textit{A reducible double-star family.}
\medskip

Recall that $D(a,b)$ denotes the double star obtained by joining the centers of $K_{1,a}$ and $K_{1,b}$. Let
\begin{equation}\label{eq:double-star-quartic}
q_{a,b}(x):= x^4-(a+b+1)x^2+ab.
\end{equation}
Differences among the $a$ leaves adjacent to one center and among the $b$ leaves adjacent to the other span a subspace of $\ker A(D(a,b))$ of dimension \((a-1)+(b-1)=a+b-2.\)
On the four-dimensional invariant subspace of vectors constant on each of the four vertex classes, the adjacency action is represented by
\[
\begin{pmatrix}
0&1&0&0\\
a&0&1&0\\
0&1&0&b\\
0&0&1&0
\end{pmatrix},
\]
whose characteristic polynomial is $q_{a,b}(x)$. Hence
\begin{equation}\label{eq:double-star-characteristic}
\phi_{D(a,b)}(x)=x^{a+b-2}q_{a,b}(x).
\end{equation}

By symmetry, let $r,s,z,u$ denote the coordinates of $(xI-A(D(a,b)))^{-1}\mathbf1$ at a left leaf, the left center, the right center, and a right leaf, respectively. Then
\[xr-s=1, \quad -ar+xs-z=1, \quad -s+xz-bu=1, \quad xu-z=1.\] Solving gives
\[
r=\frac{x(x^2+x-b)}{q_{a,b}(x)},
s=\frac{x^3+(a+1)x^2-ab}{q_{a,b}(x)},
z=\frac{x^3+(b+1)x^2-ab}{q_{a,b}(x)},
u=\frac{x(x^2+x-a)}{q_{a,b}(x)}.
\]
Since there are $a$ left leaves and $b$ right leaves, \(\Gamma_{D(a,b)}(x)=ar+s+z+bu.\) Thus
\begin{equation}\label{eq:double-star-coronal}
\Gamma_{D(a,b)}(x)=
\frac{(a+b+2)x^3+2(a+b+1)x^2-2abx-2ab}{x^4-(a+b+1)x^2+ab}.
\end{equation}

\begin{lemma}\label{lem:tree-level-determined}
For every integer $m\geq 1$, the double star $D(2m,m+1)$ is determined by its adjacency spectrum within the class of trees.
\end{lemma}

\begin{proof}
Let \(T=D(2m,m+1),\) and let $R$ be a tree cospectral with $T$. By \eqref{eq:double-star-characteristic},
\[
\phi_T(x)=x^{3m-1}
\left(x^4-(3m+2)x^2+2m(m+1)\right).\]
Hence $\operatorname{rank}A(R)=4$. Since for every forest $X$, \(\operatorname{rank}A(X)=2\nu(X),\) we have $\nu(R)=2$.

We first record the possible forms of a tree with matching number 2. Such a tree is not a star and has diameter at most 4, since a path of length five contains a matching of size 3. If $\operatorname{diam}R=3$, then \(R\cong D(c,d)\) for some positive integers $c,d$. If $\operatorname{diam}R=4$, let $v_0v_1v_2v_3v_4$ be a diametral path. The central vertex $v_2$ has no neighbor outside the chosen path: if $w\notin\{v_1,v_3\}$ were adjacent to $v_2$, then \(v_0v_1, v_2w, v_3v_4\) would form a matching of size 3. Moreover, every neighbor of $v_1$ or $v_3$ outside the path is pendant; otherwise $R$ would contain a path of length at least 5, contradicting $\operatorname{diam}R=4$. Hence $R$ is obtained from the path $v_1v_2v_3$ by attaching $c>0$ pendant vertices to $v_1$ and $d>0$ pendant vertices to $v_3$. Denote this tree by $P(c,d)$.

Suppose first that $R\cong D(c,d)$. By~\eqref{eq:double-star-characteristic},
\[
\phi_R(x)
=x^{c+d-2}
\left(
x^4-(c+d+1)x^2+cd
\right).
\]
Comparison with $\phi_T$ gives \(c+d=3m+1, cd=2m(m+1).\) Thus $c$ and $d$ are the roots of
\[
z^2-(3m+1)z+2m(m+1)
=(z-2m)(z-m-1),
\]
and hence \(\{c,d\}=\{2m,m+1\}.\) Therefore $R\cong T$.

It remains to exclude $R\cong P(c,d)$. The tree $P(c,d)$ has $c+d+3$ vertices. Its $2$-matchings consist of either one pendant edge from each end, or one of the two edges of the central $P_3$ together with a pendant edge at the opposite end. Hence
\[
\phi_{P(c,d)}(x)=
x^{c+d-1}
\left(
x^4-(c+d+2)x^2+(cd+c+d)
\right).
\]
Cospectrality with $T$ would therefore give \(c+d=3m, cd+c+d=2m(m+1),\) and hence \(cd=m(2m-1).\) Thus $c$ and $d$ would be integer roots of \(z^2-3mz+m(2m-1)=0.\)
Its discriminant is
\[
m^2+4m=m(m+4)=(m+2)^2-4.\]
This cannot be a perfect square for $m\geq1$. Indeed, if $m(m+4)=s^2$, then \[(m+2-s)(m+2+s)=4.\] The two factors are positive integers of the same parity, so both must equal $2$, which gives $m=0$, a contradiction.

Hence the second case is impossible, and therefore $R\cong T$.
\end{proof}

\begin{theorem}[A reducible double-star family]\label{thm:reducible-double-star-dgs}
Let $m\ge2$ be an integer, and suppose that neither $m$ nor $2m+2$ is a perfect square. Then $T_m:=D(2m,m+1)$ is DGS but not DS. In particular, for every integer $t\ge0$,
$D(8t+4,\,4t+3)$ is DGS but not DS.
\end{theorem}

\begin{proof}
Set \(f_m(x):=x^2-m, g_m(x):=x^2-2m-2.\) Substituting $a=2m$ and $b=m+1$ into \eqref{eq:double-star-quartic} gives \(q_{2m,m+1}(x) =x^4-(3m+2)x^2+2m(m+1)= f_m(x)g_m(x).\)
Equation~\eqref{eq:double-star-characteristic} gives
\begin{equation}\label{eq:Tm-characteristic}
\phi_{T_m}(x)= x^{3m-1}f_m(x)g_m(x).
\end{equation}
The hypotheses imply that $f_m$ and $g_m$ are distinct and irreducible over $\mathbb Q$.

Substituting $a=2m$ and $b=m+1$ into \eqref{eq:double-star-coronal} gives
\begin{equation}\label{eq:Tm-factor-coronal}
\Gamma_{T_m}(x)=\frac{m(m+1)x-2m^2}{(m+2)f_m(x)}+\frac{2(m+1)(m+3)x+8(m+1)^2}{(m+2)g_m(x)}.
\end{equation}
Both numerators are nonzero and have degree less than two. Since $f_m$ and $g_m$ are distinct irreducible polynomials, the two summands are reduced and have coprime denominators. Hence the reduced denominator of $\Gamma_{T_m}$ is $f_mg_m$. Lemma~\ref{lem:coronal-denominator} gives \(M_{T_m}(x)=f_m(x)g_m(x).\) By \eqref{eq:Tm-characteristic}, both main factors occur with exponent one in $\phi_{T_m}$, so \(\kappa_{\mathrm m}(T_m)=2.\)

By uniqueness in Theorem~\ref{thm:factor-coronal-decomposition}, \eqref{eq:Tm-factor-coronal} is the canonical factor-coronal decomposition. In particular,
\[\Gamma_{f_m}(x)=\frac{m(m+1)x-2m^2}{(m+2)f_m(x)}.\] Its zeroth factor moment is
\[
\omega_0^{(f_m)}=
\lim_{x\to\infty}x\Gamma_{f_m}(x)
=\frac{m(m+1)}{m+2}
=m-1+\frac{2}{m+2}.
\]
Since $m\geq2$, this is not an integer. But the zeroth moment of the coronal of a graph is its number of vertices. Hence $\Gamma_{f_m}$ cannot be the coronal of any graph component. Thus condition~{\rm (i)} of Theorem~\ref{thm:unified-factor-coronal-obstruction} fails under either ordering of $f_m$ and $g_m$. Corollary~\ref{cor:unified-tree-forcing} therefore implies that every generalized cospectral mate $H$ of $T_m$ is a tree. Since generalized cospectrality includes adjacency cospectrality, Lemma~\ref{lem:tree-level-determined} gives \(H\cong T_m.\) Hence $T_m$ is DGS.

In the notation of Barranca~\cite{Barranca2024} and Barranca and Barrus~\cite{BarrancaBarrus2025}, the double star $D(a,b)$ is denoted by $P_2(a,b)$. Barranca proved that
$P_2(2m,m+1)$ is not DS for every $m\geq2$; see \cite{Barranca2024}, and see also \cite[Theorem~2.3(viii)]{BarrancaBarrus2025}. Hence $T_m$ is DGS but not DS.

For $m=4t+2$ with $t\ge0$, \(m\equiv2\pmod4, 2m+2=8t+6\equiv6\pmod8.\) Squares are congruent to $0$ or $1$ modulo $4$, and to $0$, $1$, or $4$ modulo $8$. Hence neither $m$ nor $2m+2$ is a perfect square. Therefore \(T_m=D(2m,m+1)=D(8t+4,\,4t+3),\) which proves that \(D(8t+4, 4t+3)\) is DGS but not DS.
\end{proof}

For the family $D(r+3,2r+3)$ in~\cite{LinEtAl2026}, the nonzero quartic factor is irreducible over $\mathbb Q$, and the proof exploits this irreducibility. The present family is disjoint from that family, and its nonzero quartic factor is reducible: \(q_{2m,m+1}(x)=f_m(x)g_m(x).\) Thus the corresponding irreducibility argument does not apply.
Instead, the two factors are retained in the canonical factor-coronal decomposition, and $\Gamma_{f_m}$ already fails the zeroth-moment integrality condition. This excludes disconnected generalized cospectral mates without irreducibility of the quartic factor.

\section*{Declaration of competing interest}
The author declares no competing interests.

\section*{Acknowledgments}
This work was supported by the High-level Talent Research Start-up Fund of West Anhui University (Project No. WGKQ2021072).

\section*{Declaration of Use of AI Tools}
During the preparation of this work, I used OpenAI ChatGPT to assist with language refinement, manuscript organization, and supplementary checking and discussion of mathematical exposition, arguments, notation, and formula presentation. After using this tool, I reviewed, verified, and edited the content as needed and take full responsibility for the content of the publication.

\begingroup
\small  
\setlength{\bibsep}{2pt plus 1pt minus 1pt}
\bibliographystyle{cas-model2-names}
\bibliography{cas-refs}
\endgroup

\end{document}